\documentclass[11pt,a4paper]{article}
\usepackage[top=3cm, bottom=3cm, left=3cm, right=3cm]{geometry}
\usepackage[latin1]{inputenc}
\usepackage{csquotes}
\usepackage{amsmath}
\usepackage{amsthm}
\usepackage{amsfonts}
\usepackage{amssymb}
\usepackage{graphics}
\usepackage{float}
\usepackage[hang,flushmargin]{footmisc} %nonindent footnote

\usepackage{amsmath,amsthm,amssymb,graphicx, multicol, array}
\usepackage{enumerate}
\usepackage{enumitem}
\usepackage{xcolor}
\usepackage{currfile}

\usepackage[pagebackref,bookmarks,colorlinks,breaklinks]{hyperref}

\hypersetup{linkcolor=blue,citecolor=red,filecolor=blue,urlcolor=blue} 

\usepackage{tabto}

\numberwithin{equation}{section}

\usepackage{tikz}
\usepackage{pgfplots}

\usetikzlibrary{matrix}
\usepackage{mathrsfs}

\DeclareMathOperator{\ord}{ord}

\newtheorem{thm}{Theorem}[section]
\newtheorem{lem}{Lemma}[section]
\newtheorem{conj}{Conjecture}[section]
\newtheorem{exa}{Example}[section]

\newtheorem{dfn}{Definition}[section]
\newtheorem{exe}{Exercise}[section]
\newtheorem{rmk}{Remark}[section]

\newcommand{\N}{\mathbb{N}}
\newcommand{\Z}{\mathbb{Z}}
\newcommand{\Q}{\mathbb{Q}}
\newcommand{\R}{\mathbb{R}}
\newcommand{\C}{\mathbb{C}}

\newcommand{\T}{\mathbb{T}}

\usepackage{fancyhdr}
\title{Equidistribution Mod $1$ And Normal Numbers}
\date{}
\author{N. A. Carella}

\begin{document}
\maketitle

\textbf{\textit{Abstract}:} 
Let $\alpha=0.a_1a_2a_3\ldots$ be an irrational number in base $b>1$, where $0\leq a_i<b$. The number $\alpha \in (0,1)$ is a \textit{normal number} if every block $(a_{n+1}a_{n+2}\ldots a_{n+k})$ of $k$ digits occurs with probability $1/b^k$. A proof of the normality of the real number $\sqrt{2}$ in base $10$ is presented in this note. Three different proofs based on different methods are given: a conditional proof, and two unconditional proofs. 
\let\thefootnote\relax\footnote{ \today \date{} 
	}

\tableofcontents

%SSSSSSSSSSSSSSSSSSSSSSSSSSSSSSSSSSSSSSSSSSSSSSSSSSSS
%SSSSSSSSSSSSSSSSSSSSSSSSSSSSSSSSSSSSSSSSSSSSSSSSSSSS
%SSSSSSSSSSSSSSSSSSSSSSSSSSSSSSSSSSSSSSSSSSSSSSSSSSSS
%SSSSSSSSSSSSSSSSSSSSSSSSSSSSSSSSSSSSSSSSSSSSSSSSSSSS
\section{Introduction To Normal Numbers }\label{S2200} Let $\alpha=0.a_1a_2a_3\ldots$ be an irrational number in base $b>1$, where $0\leq a_i<b$. The theory of \textit{normal numbers} is centered on the distribution of the blocks $(a_{n+1}a_{n+2}\ldots a_{n+k})$ of $k$ digits in the $b$-adic expansions of the real numbers. The earliest study of normal numbers is known as the Borel conjecture. This problem investigates the distribution of the blocks of digits in the decimal expansion of the number $\sqrt{2}=1.414213562373 \ldots$. 

\begin{dfn} \label{dfn2200.001} {\normalfont An irrational number $\alpha \in \R$ is a \textit{normal number in base} $b>1$ if any sequence of $k$-digits in the $b$-adic expansion occurs with probability $1/b^k$. Further, the number is called \textit{absolutely normal} if it is a normal number in every base.
	}
\end{dfn}

A normal number in any base is an irrational number, but an irrational number is not necessarily normal. The simplest, and best known construction technique of normal numbers is the integers concatenation method. In base $b=10$, the best known examples are the followings numbers. The concatenation of the consecutive integers:
\begin{equation}\label{eq2200.400}
	C_{0}=\sum_{n\geq 1}\frac{n}{10^{a_0(n)}}=0.12345678910111213151617181920\ldots,
\end{equation}
where $a_0(n)=n+\sum_{1\leq k\leq n}[\log_{10} k]$. The concatenation of the consecutive primes: 
\begin{equation}\label{eq2200.410}
	C_{1}=\sum_{n\geq 1}\frac{p_n}{10^{a_1(n)}}=0.235711131719232931374143475359\ldots,
\end{equation}
where $a_1(n)=n+\sum_{1\leq k\leq n}[\log_{10}p_k ]$, and $p_n$ is the $n$th prime in increasing order. The concatenation of the consecutive squares:
\begin{equation}\label{eq2200.420}
	C_{2}=\sum_{n\geq 1}\frac{n^2}{10^{a_2(n)}}=0.149162536496481100121144169196\ldots,
\end{equation}
where $a_2(n)=n+\sum_{1\leq k\leq n}[\log_{10}k^2 ]$, and many other similar numbers. The integers concatenation method generalizes to any base $b>1$, and to certain infinite sequences of integers. The basic proofs are provided in \cite{CD1933}, \cite[Theorem 1]{CE1946}, et alii. Another class of normal numbers is defined by series of the forms
\begin{equation}\label{eq2200.430}
	C_{b,c}(s)=\sum_{n\geq 1}\frac{1}{c^nb^{c^n+s}},
\end{equation}
where $\gcd(b,c)=1$, and $s\in \R$, see \cite{SR1972}, \cite{BC2002}, et alii. More complex constructions and algorithms for generating normal numbers are developed in \cite{AB2017}, \cite{BY2019}, et alii. However, there are no known normal numbers in closed forms such as $\sqrt{2}$, $e$, $\pi$, $\log 2$, $\gamma$, ..., et cetera in any base $b\geq2$. This note contributes the followings conditional results using two different methods.  

\begin{thm} \label{thm2200.300} The irrational number $\sqrt{2} \in \R$ is a normal number in base $p\geq2$. In particular, $\sqrt{2}$ is simply normal number in base $10$. Hence, the decimal expansion $$\sqrt{2} = 1.4142135623730950488016887242096980785696718753769480731766797379\ldots$$ 
	contains infinitely many digit $0$, infinitely many digit $1$, infinitely many digit $2$, et cetera.
\end{thm}

\begin{thm} \label{thm2200.400} Assume the GRH. Then, the irrational number $\pi \in \R$ is a normal number in base $10$.
\end{thm}

The preliminary foundational results required to prove these results are covered in Section \ref{S5588} to Section \ref{S7711}. The proofs are presented in chronological order as discovered. 
The conditional proof of the normality of $\sqrt{2}$ in base $10$, Theorem \ref{thm2200.300}, is exhibited in Section \ref{S8877}, and the conditional proof of the normality of $\pi$ in base $10$, Theorem \ref{thm2200.400}, is exhibited in Section \ref{S7799}. The first unconditional proof of the normality of $\sqrt{2}$ in base $p\geq2$ is a corollary of Theorem \ref{thm9900.900} in Section \ref{S9900} and the second unconditional proof of the normality of $\sqrt{2}$ in base $p\geq2$ is a corollary Theorem \ref{thm9900.900} in Section \ref{S9900}.\\

Surveys of some of the literature on normal numbers appear in \cite{KD2006}, et alii. A survey of some of the literature on the number $\pi$ appears in \cite{BJ2014}, and its irrationality is proved in 
\cite[Corollary 2.6]{NI1956}.

%SSSSSSSSSSSSSSSSSSSSSSSSSSSSSSSSSSSSSSSSSSSSSSSSSSSS
%SSSSSSSSSSSSSSSSSSSSSSSSSSSSSSSSSSSSSSSSSSSSSSSSSSSS
%SSSSSSSSSSSSSSSSSSSSSSSSSSSSSSSSSSSSSSSSSSSSSSSSSSSS
%SSSSSSSSSSSSSSSSSSSSSSSSSSSSSSSSSSSSSSSSSSSSSSSSSSSS
%\newpage
\section{Notation}\label{S2266N}
The sets $\N=\{0,1,2,3,\ldots\}$ and $\Z=\{\ldots, -2,-1,0,1,2,\ldots\}$ are the set of nonnegative integers and the set of integers. The symbols $\Q$, $\R$, and $\C$ denote the sets of rational numbers, the set of real numbers, and the set of complex numbers respectively.\\

Let $f, g: [x_0,\infty]\longrightarrow \R$ be a pair of functions, and assume $g(x)>0$. The little o notation is defined by
\begin{equation}\label{eq2266N.100}
	f(x)=o(g(x))  \quad  \Longleftrightarrow  \quad|f(x)|\leq c g(x)
\end{equation}
for any constant $c>0$ as $x\to\infty$. The negation of the little o notation is defined by
\begin{equation}\label{eq2266N.110}
	f(x)\ne o(g(x))  \quad  \Longleftrightarrow  \quad f(x)=\Omega_{\pm}(g(x)).
\end{equation}
This has the explicit form
\begin{equation}\label{eq2266N.115}
	f(x)=\Omega_{\pm}(g(x)) \quad  \Longleftrightarrow  \quad |f(x)|\geq c g(x).
\end{equation}
for some constant $c>0$ as $x\to\infty$.
The big O notation is defined by
\begin{equation}\label{eq2266N.120}
	f(x)=O(g(x))  \quad  \Longleftrightarrow  \quad|f(x)|\leq c g(x)
\end{equation}
for some constant $c>0$ as $x\to\infty$.

%SSSSSSSSSSSSSSSSSSSSSSSSSSSSSSSSSSSSSSSSSSSSSSSSSSSS
%SSSSSSSSSSSSSSSSSSSSSSSSSSSSSSSSSSSSSSSSSSSSSSSSSSSS
%SSSSSSSSSSSSSSSSSSSSSSSSSSSSSSSSSSSSSSSSSSSSSSSSSSSS
%SSSSSSSSSSSSSSSSSSSSSSSSSSSSSSSSSSSSSSSSSSSSSSSSSSSS
%\newpage
\section{Lacunary and Nonlacunary  Sequences}\label{S3300}
A \textit{lacunary} sequence is a sparce or thin sequence of integers. A few of the properties of lacunary sequences are recorded in this  section.

\begin{dfn}\label{dfn4400.100}{\normalfont
		A pair of integers $p>1$ and $q>1$ are \textit{multiplicative independent} over the integers if $p^a\ne q^b$ for all nonzero integers $a,b\in \Z^{\times}$. 
	}
\end{dfn} 
\vskip .15 in 
\begin{dfn} \label{dfn3300.400}{\normalfont Let $\mathscr{U}\subset \N$ be a subset of integers. The subset is \textit{lacunary} if and only if it is generated by a sequence of powers $\mathscr{U}=\{u^n: n\geq0\}$, where $u>1$ is an integer. In particular, $$\lim_{n\to \infty}\frac{u_{n+1}}{u_n}>1.$$ Otherwise, it is \textit{nonlacunary} and it is generated by two or more powers.
	}
\end{dfn}

The basic nonlacunary sequences can be classified into two different forms: additive and multiplicative. 

\begin{dfn} \label{dfn3300.410}{\normalfont An \textit{additive nonlacunary sequence} is of the form
		\begin{equation}\label{eq3300.410A}
			\mathscr{U}=	\{u_{m,n}=p_1^nq_1^m+p_2^nq_2^m+\cdots +p_d^nq_d^m:m,n\geq1\},
		\end{equation}
		and a \textit{multiplicative nonlacunary sequence} is of the form
		\begin{equation}\label{eq3300.410B}
			\mathscr{V}=	\{v_{m,n}=p_1^nq_1^mp_2^nq_2^m\cdots p_d^nq_d^m:m,n\geq1\},
		\end{equation}
		where the generators $p_i$ and $q_i$ are multiplicative independent integers, and $d\geq1$ is fixed dimension. 
	}
\end{dfn}

The analytic method for determining the whether or not a semigroup is lacunary has a simple form.

\begin{lem}\label{lem3300.400} A semigroup $\mathscr{L}=\{u_n: n\geq1\}\subset \N$ is lacunary if and only if $$\lim_{n\to \infty}\frac{u_{n+1}}{u_n}>1.$$
\end{lem} 
\begin{exa}\label{exa3300.400}{\normalfont Consider the generators $u=2$ and $v=3$. Then
\begin{enumerate}
\item $\displaystyle \mathscr{L}=\{2^n: n\geq0\}\subset \N,$ is a lacunary semigroup.\\
\item $\displaystyle \mathscr{M}=\{2^m3^n: m,n\geq0\}\subset \N,$ is a nonlacunary semigroup.
\end{enumerate}
	
}
\end{exa}

\begin{exa}\label{exa3377.410}{\normalfont The simple limit test classifies of the followings subsets of integers.
\begin{enumerate}
\item $\displaystyle \mathscr{M}=\{u_n=n^7: n\geq0\}\subset \N,$ is a nonlacunary semigroup.\\
\item $\displaystyle \mathscr{P}=\{u_n=p^9: p\geq2\}\subset \N,$ with $p$ prime, is a nonlacunary semigroup.
\end{enumerate}

	}
\end{exa}

%SSSSSSSSSSSSSSSSSSSSSSSSSSSSSSSSSSSSSSSSSSSSSSSSSSSS
%SSSSSSSSSSSSSSSSSSSSSSSSSSSSSSSSSSSSSSSSSSSSSSSSSSSS
%SSSSSSSSSSSSSSSSSSSSSSSSSSSSSSSSSSSSSSSSSSSSSSSSSSSS
%SSSSSSSSSSSSSSSSSSSSSSSSSSSSSSSSSSSSSSSSSSSSSSSSSSSS
\section{Dense Sets}\label{S3377}
Several techniques for generating dense sets of real numbers are described in this section.
\begin{thm} \label{thm3377.800}{\normalfont (\cite[Theorem IV.1]{FH1967})} If $\mathscr{M}\subset \N$ is a nonlacunary semigroup of integers, and $\alpha$ irrational number, then, the subset of real numbers $\{\alpha n:n\in\mathscr{M}\}$ is dense in the torus. In particular, the closure
	$$\overline{\{\alpha n:n\in\mathscr{M}\}}=\R/\Z.$$
\end{thm}
\vskip .15 in
\begin{exa}\label{exa3377.430}{\normalfont For a pair of multiplicative independent integers $p$ and $q$, and an irrational number $\alpha$, the double sequence of real numbers $\{\beta_{m,n}=p^mq^n\alpha:m,n\geq1 \}$ is dense in the torus $\T=\R/\Z$.
	}
\end{exa}
\vskip .15 in
A few generalizations of this result have been established in the literature. One of these results is stated here.\\

\begin{thm} \label{thm3377.850}{\normalfont (\cite[Theorem 1.2]{KB1999})} Let $p_i, q_i \in\N$ with $1 < p_i < q_i$ for $i = 1,... ,d$ and
	assume that $p_1 \leq p_2\leq\cdots \leq  p_k$. Assume that the pairs $p_i$, $q_i$ are multiplicative
	independent for $i = 1,... ,d$. Then for distinct $\alpha_1,... ,\alpha_d\in \T$ with at least one
	$\alpha_i\notin\Q$, the subset of real numbers
	$$\mathscr{A}=	\{\beta_{m,n}=p_1^nq_1^m\alpha_1+p_2^nq_2^m\alpha_2+\cdots +p_d^nq_d^m\alpha_d:m,n\geq1\}$$
	is dense in the torus $\T=\R/\Z$. In particular, the closure
	$$\overline{\{\alpha\in\mathscr{A}\}}=\T.$$
\end{thm}

A result linking the space of countable dense sets to the space of equidistributed sequences will be required to complete the proof of the main result. 

\begin{lem}\label{lem3377.750} Let $\mathscr{D}\subset [0,1]$ be a countable subset of real numbers. Then, the followings statements are equivalent.  
	\begin{enumerate} [font=\normalfont, label=(\roman*)]
		\item The countable set $\displaystyle \mathscr{D}=\{x_1,x_2,\ldots\}$ is dense in $[0,1]$.
		\item There exists a permutation $\displaystyle \sigma\mathscr{D}=\{y_1,y_2, \ldots \}$ of the countable set $\displaystyle \mathscr{D}$ such that the sequence $\{y_n=\sigma(x_n):n\geq1\}$ is equidistributed in $[0,1]$.
	\end{enumerate}
\end{lem} 

%SSSSSSSSSSSSSSSSSSSSSSSSSSSSSSSSSSSSSSSSSSSSSSSSSSSS
%SSSSSSSSSSSSSSSSSSSSSSSSSSSSSSSSSSSSSSSSSSSSSSSSSSSS
%SSSSSSSSSSSSSSSSSSSSSSSSSSSSSSSSSSSSSSSSSSSSSSSSSSSS
%SSSSSSSSSSSSSSSSSSSSSSSSSSSSSSSSSSSSSSSSSSSSSSSSSSSS
\section{Results for Pisot and Salem Numbers }\label{S2323}
A distinguished collection of algebraic numbers is defined below. The recent papers, \cite{SR1944}, \cite{BD1978} et alii, introduce some of the properties of these numbers. 
\begin{dfn}\label{dfn2323.200}{\normalfont
		Let $\theta\ne0$ be a root of an irreducible polynomial $f(z)\in \Z[z]$ of degree $\deg f=d\geq2$. Then,
		\begin{enumerate}
			\item The real number $\theta>1$ is called a \textit{Pisot number} if the conjugates roots have absolute value $|\theta_i|<1$ for $i=1,2,3,\ldots,d-1$.\\
			\item The real number $\theta>1$ is called a \textit{Salem number} if the conjugates roots have absolute value $|\theta_i|\leq1$, and at least one has $|\theta_j|=1$ for $i=1,2,3,\ldots,d-1$.\\
		\end{enumerate}
	}
\end{dfn}

\begin{lem}\label{lem2323.300}The sequence powers $\{\theta^n:n\geq1\}$ of a Pisot number is approximates a subsequence of integers exponentially fast.
\end{lem} 
\begin{proof}[\textbf{Proof}]By the Newton identity, a sum of conjugate algebraic integers
	\begin{equation}\label{eq2323.310}
		\theta^n+\theta_1^n+\theta_2^n+\cdots+\theta_{d-1}^n
	\end{equation}	
	is an integer for any integer $n\geq1$. By the definition of a Pisot number $|\theta_i|<1$, and $|\theta|>1$. Therefore, the rearrange sum
	\begin{equation}\label{eq2323.320}
		\theta^n=-\theta_1^n-\theta_2^n-\cdots-\theta_{d-1}^n
	\end{equation}
	converges to an integer at exponential rate	as $n\to \infty$. 
\end{proof}  
\begin{exa}\label{exa2323.210} {\normalfont The first Pisot number seems to be the real root
		$$\theta=\sqrt[3]{((9+\sqrt{69})/18)}+\sqrt[3]{((9-\sqrt{69})/18)}= 1.324717957244746\ldots
		$$ of the polynomial $x^3-x-1$. The sum of the remaining conjugate roots is		$$\theta_1+\theta_2=2r\cos 2\pi\omega,$$where $r<1$ and $\omega \in (0,1)$, see Figure \ref{fig2323.110} below.\\
		
		The corresponding sequence 
		$$x_n=\theta^n+\theta_1^n+\theta_2^n$$ is exponentially close to a subsequence of integers as $n\to \infty$.
		
	}
	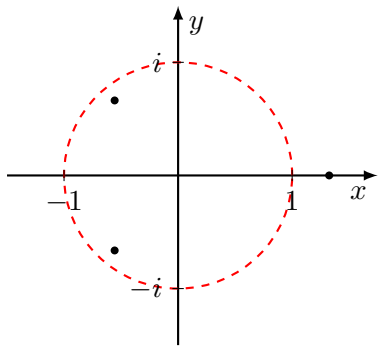
\begin{figure}[H]\label{fig2323.110}
		\begin{center}
			\begin{tikzpicture}[scale=1.5]
				\draw[thick,-latex] (-1.5,0) -- (1.75,0) node[below left] {$x$};
				\draw[thick,-latex] (0,-1.5) -- (0,1.5) node[below right] {$y$};
				\draw[red,thick,dashed] (0,0) circle (1cm);
				\draw (-1,0.05) -- (-1,-0.05) node[below]{$-1$};
				\draw (1,0.05) -- (1,-0.05) node[below]{$1$};
				\draw (0.05,1) -- (-0.05,1) node[left]{$i$};
				\draw (0.05,-1) -- (-0.05,-1) node[left]{$-i$};
				\foreach \X/\Y in {-130/0.8656,130/0.8656,0/1.3247}
				{\fill (\X:\Y) circle[radius=1pt];}
			\end{tikzpicture}
		\end{center}
		\caption{The roots of $x^3-x-1$}
	\end{figure}
\end{exa}

\begin{exa}\label{exa2323.250} {\normalfont The first Salem number seems to be the real root
		$$\theta=1.266361236713076
		\ldots
		$$ of the Lehmer polynomial $x^{10}+x^9-x^7-x^6-x^5-x^4-x^3+x+1$. The sum of the conjugate roots is
		$$\theta+\theta_1+\theta_2+\cdots+\theta_8+\theta_9=\theta+\theta_1+2r_1\cos 2\pi \omega_1+\cdots+2r_4\cos 2 \pi\omega_4,$$where $\theta_1=\theta^{-1}<1$, $|\theta_i|=r_i\leq1$ and $\omega_i \in (0,1)$ for $i\in \{2,3,\ldots,9\}$. The first pair of conjugate roots $\theta_2+\theta_3=2r_1\cos \omega_1$, and so on.\\
		
		The corresponding sequence 
		\begin{eqnarray}\label{???}
			x_n&=&\theta^n+\theta_1^n+\theta_2^n+\cdots +\theta_8^n+\theta_9^n\nonumber\\
			&=&\theta^n+\theta_1^n+r_1^n\cos \omega_1n+\cdots+r_4^n\cos \omega_4n,\nonumber
		\end{eqnarray}
		where at least one $r_i=1$, is exponentially close to a subsequence of integers as $n\to \infty$.
	}
\end{exa}

\begin{thm}\label{thm2323.700} If $\theta>1$ is a Salem number of degree $d=2m\geq2$, then sequence powers
	\begin{eqnarray}\label{eq2323.700}
		\theta^n+\theta^{-n}&=&-\theta_2^n-\cdots -\theta_{d-2}^n-\theta_{d-1}^n\nonumber\\
		&=&-2\cos 2\pi\omega_1n-\cdots-2\cos 2\pi\omega_{m-1}n,\nonumber
	\end{eqnarray}	
	where $\omega_i\in(0,1)$, is dense, but not equidistributed on the unit interval $(0,1)$.
\end{thm}
\begin{proof}[\textbf{Proof}] Under this condition all the complex roots $\theta_2,\theta_3,\ldots,\theta_{2d-1}$ are on the unit circle, and the two real roots are $\theta>1$ and $\theta_1=\theta^{-1}<1$. A proof that the sequence is not equidistributed is derived from the relation
	\begin{eqnarray}\label{eq2323.710}
		e^{i2\pi k\left (\theta^n+\theta^{-n}\right)}&=&e^{-i2\pi k \left (\theta_2^n+\cdots +\theta_{d-2}^n+\theta_{d-1}^n\right) }\nonumber\\
		&=&e^{-i2\pi k \left (2\cos 2\pi\omega_1n+\cdots+2\cos 2\pi\omega_{m-1}n\right)} ,\nonumber\\
		&=&\prod_{1\leq s\leq d-1}e^{-i4\pi k \cos 2\pi\omega_sn} \nonumber,
	\end{eqnarray}	
	and the Bessel function identity
	\begin{equation}\label{eq2323.720}
		J_0(-4\pi k)=\int_0^1 e^{-i4\pi k \cos 2\pi\omega(t)}dt.
	\end{equation}
	A complete and short proof is given in \cite{PS1964}.
\end{proof}

Recent refinements of this results are available in the literature, see \cite{AT2004}, \cite{DR2008}.

%SSSSSSSSSSSSSSSSSSSSSSSSSSSSSSSSSSSSSSSSSSSSSSSSSSSS
%SSSSSSSSSSSSSSSSSSSSSSSSSSSSSSSSSSSSSSSSSSSSSSSSSSSS
%SSSSSSSSSSSSSSSSSSSSSSSSSSSSSSSSSSSSSSSSSSSSSSSSSSSS
%SSSSSSSSSSSSSSSSSSSSSSSSSSSSSSSSSSSSSSSSSSSSSSSSSSSS
%\newpage
\section{Equidistribution Criteria}\label{S2266}
The standard Weyl criterion claims that \eqref{eq2266.500} is true for all integer parameter $k\ne0$, see \cite[Theorem 2.1]{KN1974}. This result is extended to all rational parameter $k=r\ne0$.
\begin{thm}\label{thm2266.500} {\normalfont (Extended Weyl Criterion)} The sequence of real number $\{u_n:n\geq 1\}$  is uniformly distributed modulo $1$ if and only if 
	\begin{equation}\label{eq2266.500} 
		\lim_{x\to \infty} \frac{1}{x}\sum_{n\leq x}e^{i 2 \pi  ku_ n}=0
	\end{equation}
	for any fixed rational number $k\ne0$.
\end{thm}

\begin{proof}Without loss in generality assume $\{u_n=\alpha n :n\geq 1\}$, where $\alpha\ne0$ is an irrational number. Summing the exponential sum yields
	\begin{eqnarray}\label{eq2266.550} 
		\left | \sum_{n\leq x}e^{i 2 \pi \alpha k n}\right |&=&\left |\frac{e^{i 2 \pi \alpha k (x+1)}-1}{e^{i 2 \pi \alpha k}-1}\right |\\
		&\leq &\frac{1}{\left |\sin( \pi \alpha k)\right |}\nonumber.
	\end{eqnarray}
	Since $\alpha\ne0$ is an irrational number, the product representation of the sine function shows that
	\begin{equation}\label{eq2266.560} 
		\sin(\pi \alpha k)=\pi \alpha k \prod_{n\geq 1}\left ( 1-\frac{k^2 \alpha^2}{n^2}\right)\ne 0
	\end{equation}
	for any fixed rational number $k\ne0$, the last inequality \eqref{eq2266.550} is bounded by a constant.
\end{proof}
\begin{thm} \label{thm2266.200} {\normalfont (\cite{KJ1935})} For any real number $\alpha\ne0$, and almost all irrational $\theta \in \R^{\times}$, the sequence $\displaystyle \{\alpha\theta^n:n\geq1\}$ is equidistributed modulo $1$. 
\end{thm}
The best known exceptions to this theorem are documented in Theorem \ref{thm2323.700}.
%SSSSSSSSSSSSSSSSSSSSSSSSSSSSSSSSSSSSSSSSSSSSSSSSSSSS
%SSSSSSSSSSSSSSSSSSSSSSSSSSSSSSSSSSSSSSSSSSSSSSSSSSSS
%SSSSSSSSSSSSSSSSSSSSSSSSSSSSSSSSSSSSSSSSSSSSSSSSSSSS
%SSSSSSSSSSSSSSSSSSSSSSSSSSSSSSSSSSSSSSSSSSSSSSSSSSSS
%\newpage
\section{Equidistribution Criteria for Normal Numbers }\label{S2270}
\begin{thm} \label{thm2270.150}{\normalfont (Wall)} An irrational number $\alpha \in \R$ is a normal number in base $p\geq2$ if and only if the sequence $\{\alpha p^n: n \geq 1\}$ is uniformly distributed modulo $1$.
\end{thm}
The proof of this criterion appears in \cite[Theorem 8.15]{NI1956}.

%SSSSSSSSSSSSSSSSSSSSSSSSSSSSSSSSSSSSSSSSSSSSSSSSSSSS
%SSSSSSSSSSSSSSSSSSSSSSSSSSSSSSSSSSSSSSSSSSSSSSSSSSSS
%SSSSSSSSSSSSSSSSSSSSSSSSSSSSSSSSSSSSSSSSSSSSSSSSSSSS
%SSSSSSSSSSSSSSSSSSSSSSSSSSSSSSSSSSSSSSSSSSSSSSSSSSSS
%\section{Results for Normal Numbers }\label{S1717}

\begin{thm} \label{thm2270.300} Given  real number $\alpha \in \R^{\times}$, the following statements are valid.
	\begin{enumerate}[font=\normalfont, label=(\roman*)]
		\item If $\alpha $ is irrational, then $\displaystyle r\alpha+s$ is irrational for all rational scale $r,s \in\mathbb{Q}^{\times}$.
		\item If $\alpha $ is normal in base $b$, then $\displaystyle r\alpha+s$ is normal in base $b$ for all rational scale $r,s \in\mathbb{Q}^{\times}$.
	\end{enumerate}	
\end{thm}
\begin{proof}(ii) This follows from the Wall criterion, Theorem \ref{thm2270.150}.
\end{proof}
%SSSSSSSSSSSSSSSSSSSSSSSSSSSSSSSSSSSSSSSSSSSSSSSSSSSS
%SSSSSSSSSSSSSSSSSSSSSSSSSSSSSSSSSSSSSSSSSSSSSSSSSSSS
%SSSSSSSSSSSSSSSSSSSSSSSSSSSSSSSSSSSSSSSSSSSSSSSSSSSS
%SSSSSSSSSSSSSSSSSSSSSSSSSSSSSSSSSSSSSSSSSSSSSSSSSSSS
%\section{Primes And Primitive Roots In Short Intervals}\label{S3366}

%\begin{lem} \label{lem2266.600} Let $x\geq 1$ be a large number, and let $[x,x+x^{3/5}]$ be aan interval. Then, the interval contains $\gg x^{3/5}/\log x$ prime numbers such that

%\begin{equation}\label{eq3366.610} 
%\{10^n\equiv m \bmod p:n<p\}=\mathbb{F}_p^{\times}.
%\end{equation}
%\end{lem}
%\begin{proof}The existence of primes in the short interval $[x,x+x^{3/5}]$ is well known, see \cite{IA1985}, \cite{KA1994}, et alii for the proof. By Hooley theorem, see \cite{HC1967}, there are 
%\begin{equation}\label{eq3366.620} 
%a_0\pi(x+x^{3/5})-a_0\pi(x)\gg \frac{x^{3/5}}{\log x}
%\end{equation}
%such primes as $x\to \infty$, where the density of these primes is $a_0=\prod_{p\geq 2}\left ( 1-\frac{1}{p(p-1)}\right )=0.373955\ldots$, the precise derivation of this density appears %in \cite[p. 220]{HC1967}.
%\end{proof}

%SSSSSSSSSSSSSSSSSSSSSSSSSSSSSSSSSSSSSSSSSSSSSSSSSSSS
%SSSSSSSSSSSSSSSSSSSSSSSSSSSSSSSSSSSSSSSSSSSSSSSSSSSS
%SSSSSSSSSSSSSSSSSSSSSSSSSSSSSSSSSSSSSSSSSSSSSSSSSSSS
%SSSSSSSSSSSSSSSSSSSSSSSSSSSSSSSSSSSSSSSSSSSSSSSSSSSS
%\newpage
\section{Equidistribution Criteria for Double Sequences}\label{S4400}
There are several criteria for the equidistribution of double sequences. The criterion stated in Theorem \ref{thm4400.700}, and the well known Furstenberg criterion, see Theorem \ref{thm3300.800} are used in proof of the main result. Basically, these are extensions of the Weyl criterion in Theorem \ref{thm2266.500} and or Wall criterion in Theorem \ref{thm2270.150}. \\

\begin{thm} \label{thm4400.700}{\normalfont (\cite[Theorem 2.9]{KN1974})} The double sequence of real numbers $\{\beta_{m,n}:m,n\geq1 \}$ is uniformly distributed modulo $1$ if and only if
	$$\lim_{x,y \to \infty}\frac{1}{xy}\sum_{ m\leq y,}\sum_{ n\leq x}e^{i2 \pi k\beta_{m,n}}=0$$	
	for all integers $k\ne0$.
\end{thm}

%SSSSSSSSSSSSSSSSSSSSSSSSSSSSSSSSSSSSSSSSSSSSSSSSSSSS
%SSSSSSSSSSSSSSSSSSSSSSSSSSSSSSSSSSSSSSSSSSSSSSSSSSSS
%SSSSSSSSSSSSSSSSSSSSSSSSSSSSSSSSSSSSSSSSSSSSSSSSSSSS
%SSSSSSSSSSSSSSSSSSSSSSSSSSSSSSSSSSSSSSSSSSSSSSSSSSSS
%\newpage
\section{Equidistribution Criteria for Multivariable Polynomials}\label{S4405}
The equidistribution criteria for the sequences of real numbers $\{x_{n_1,\ldots,n_d}=f(n_1,\ldots ,n_d)\alpha:n_i\geq1\}$ generated by multivariable polynomials $f(t_1,\ldots ,t_d)\in\Z[t_1,\ldots ,t_d]$ are similar to the criteria for polynomial of a single variable.

\begin{thm} \label{thm4405.700} Let $\alpha$ be an irrational number, and let $f(t_1,\ldots ,t_d)$ be a monotonically increasing polynomial over the integers. Then, the sequence of real numbers $\{f(t_1,\ldots ,t_d)\alpha:n_i\geq1 \}$ is uniformly distributed modulo $1$. 
\end{thm}

\begin{proof}[\textbf{Proof}] The proof for the simplest case of additive multivariable polynomial $f(t_1,t_2)=t_1^a+t_2^b$, where $a\geq1$ and $b\geq1$, is provided here. Consider the sequence of real numbers
	\begin{equation}\label{eq4405.700}
\{\beta_{m,n}=(m^a+n^b)\alpha:m,n\geq1\}.
\end{equation}	
	
An application of Theorem \ref{thm4400.700} yields
	\begin{equation}\label{eq4405.710}
		\frac{1}{x^2}\;	\sum_{ m\leq x,}	\sum_{ n\leq x}e^{i2 \pi k\beta_{m,n}}=\frac{1}{x^2}\;	\sum_{ m\leq x,}	\sum_{ n\leq x}e^{i2 \pi k \left(m^a+n^b\right )\alpha}=o(1)
	\end{equation}
	for any parameter $k\ne0$. \\
	
	Rearranging the exponential sum as a product yields the decomposition 
	\begin{eqnarray}\label{eq4405.720}
		\frac{1}{x^2}\;	\sum_{ m\leq x,}	\sum_{ n\leq x}e^{i2 \pi k \left(m^a+n^b\right )\alpha}&=&\left(\frac{1}{x}\;	\sum_{ m\leq x}	e^{i2 \pi km^a\alpha} \right)\times \left(\frac{1}{x}\;	\sum_{ n\leq x}e^{i2 \pi k n^b\alpha}\right )\nonumber\\
		&=&o(1).
	\end{eqnarray}
Clearly, these exponential sums satisfy the conditions
	\begin{equation}\label{eq4405.730}
		\frac{1}{x}\;	\sum_{ m\leq x}	e^{i2 \pi km^a\alpha} =o(1)\qquad \text{ and }\qquad \frac{1}{x}\;	\sum_{ n\leq x}e^{i2 \pi k n^b\alpha}
		=o(1).
	\end{equation}
	
	Therefore, the sequence of real numbers \eqref{eq4405.700} is equidistributed modulo $1$. The general case for any multivariable polynomial is similar.
\end{proof}
\vskip .15 in
\begin{exa}\label{exa4405.600}{\normalfont The followings sequences of real numbers are equidistributed modulo $1$.
		\begin{enumerate}
			\item For any irrational $\alpha$, the subset of real numbers $\displaystyle \mathscr{M_1}=\{v_{m,n}=(m^2+n^3)\alpha: m,n\geq0\}$. \\
			
\item For any irrational $\alpha$, the subset of real numbers $\displaystyle \mathscr{M_2}=\{v_{m,n}=(m^4n^5+m^9+n^3)\alpha: m,n\geq0\}$. \\
		\end{enumerate}
	}
\end{exa}

%SSSSSSSSSSSSSSSSSSSSSSSSSSSSSSSSSSSSSSSSSSSSSSSSSSSS
%SSSSSSSSSSSSSSSSSSSSSSSSSSSSSSSSSSSSSSSSSSSSSSSSSSSS
%SSSSSSSSSSSSSSSSSSSSSSSSSSSSSSSSSSSSSSSSSSSSSSSSSSSS
%SSSSSSSSSSSSSSSSSSSSSSSSSSSSSSSSSSSSSSSSSSSSSSSSSSSS
\section{Rational Prime Approximations}\label{S4466}
The quantity $\left | \left | x \right | \right |=\min_{n\in \Z}\{|x-n|\}\geq0 $ defines the least distance to an integer. Using this notation, the rational number approximation $\left |q\alpha -m \right | < q^{-1 }$  can be written in the form $\left | \left | q\alpha \right | \right | < q^{-1 }$, where $q\geq1$ is an integer. The  rational prime approximations is a refinedment to primes denominators $q=p$. 

\begin{thm} {\normalfont (\cite[Theorem 1]{MK2009})} \label{thm2266.600} Let $\varepsilon>0$ and $\tau = 1/3-\varepsilon$. Then, there exist infinitely many primes $p$ such
	that
	\begin{equation}\label{eq4466.610} 
		\left | \left | p\alpha \right | \right | < p^{-1/3+\varepsilon }.
	\end{equation}
\end{thm}
For applications to normal numbers, an additional property is essential.
\begin{conj} \label{conj2266.700} Let $\varepsilon>0$ and $\tau = 1/3-\varepsilon$. Then, there exist infinitely many primes $p$ such
	that
	\begin{equation}\label{eq4466.700} 
		\left | \left | p\alpha \right | \right | < p^{-1/3+\varepsilon },
	\end{equation}
	and the integer $10$ has a large multiplicative order $\ord_p10\gg p^{\delta}$, where $\delta>0$ is a small number.
	
\end{conj}
%SSSSSSSSSSSSSSSSSSSSSSSSSSSSSSSSSSSSSSSSSSSSSSSSSSSS
%SSSSSSSSSSSSSSSSSSSSSSSSSSSSSSSSSSSSSSSSSSSSSSSSSSSS
%SSSSSSSSSSSSSSSSSSSSSSSSSSSSSSSSSSSSSSSSSSSSSSSSSSSS
%SSSSSSSSSSSSSSSSSSSSSSSSSSSSSSSSSSSSSSSSSSSSSSSSSSSS
%\newpage
\section{Random Rational Approximations}\label{S5588}
The one-to-one correspondence between the factional parts of the sequence $x_n=n\log 10 +\log \pi$ and $y_n=\pi10^n $, where $n\geq 1$, via the map $x_n \longrightarrow y_n=e^{x_n}$ clues to the uniform distribution of both sequences. To explicate the relationship between all the sequences and fractional parts consider the sets
\begin{align}
	X&=\left\{x_n=n\log 10 +\log \pi:n\geq1 \right\},\\   
	Y&=\left\{y_n= \pi 10^n:n\geq1 \right\},\nonumber\\
	U&=\left\{\{x_n\}:n\geq1 \right\}, \nonumber\\
	V&=\left\{\{y_n\}:n\geq1\right\}, \nonumber
\end{align}
and the correspondence diagram.
\begin{center}
	\begin{tikzpicture}
		\matrix (m) [matrix of math nodes,row sep=3em,column sep=4em,minimum width=2em]
		{
			X & Y \\
			U & V\\};
		\path[-stealth]
		(m-1-1) edge node [left] {$\rho$} (m-2-1)
		(m-1-1) edge node [above] {$e^{x_n}$} (m-1-2)
		(m-2-1) edge node [below] {$e^{\{x_n\}}$} (m-2-2) 
		(m-1-2) edge node [right] {$\rho$} (m-2-2);
	\end{tikzpicture}
\end{center}
The class function $\rho :\R \longrightarrow [0,1]$, refer to the above diagram, assigns the fractional parts $\rho(x_n)=\{x_n\}$ and $\rho(y_n)=\{y_n\}$ of the corresponding sequences $x_n=n\log 10 +\log \pi$ and $y_n=e^{x_n}=\pi 10^n$, respectively, where $n\geq 1$.\\

The equidistribution of the sequence $x_n$ is established in the Lemma below.
\begin{lem}\label{lem5588.800} The sequence of real number $x_n=n\log 10 +\log \pi$, where $n\geq 1$, is uniformly distributed modulo $1$.
\end{lem}

\begin{proof}A routine application of Theorem \ref{thm2266.500}, see also \cite[Theorem 2.1]{KN1974}. 
\end{proof}
The conditional proof of the equidistribution of the sequence $y_n=e^{x_n}=\pi10^n $ is significantly longer. Some of the required foundational results are established here.\\

The result below shows that for each integer $n\geq 1$, the fractional part $\{\pi10^n\}$ is contained in some random subinterval of the form
\begin{equation}\label{eq5588.195} 
	\left [\frac{r_n}{q_k}+\frac{1}{q_k^{\mu-1}},\frac{s_n}{q_k}+\frac{c_n}{q_k^2}\right )\subseteq [0,1],
\end{equation}

where $0\leq r_n,s_n,c_n\leq q_k$. In this application, the lower bound $\mu(\pi)\geq2$ of the irrationality measure of the real number $\pi$ is sufficient, see \cite[p. \  556]{WM2000} for the definition of this quantity.

\begin{lem} \label{lem5588.100}   If $p_k/q_k$ is the $k$th convergent of the real number $\pi$, then, for each $n\leq q_k$, the fractional part $\{\pi 10^n \}$ of the real number $\pi10^n$ satisfies one or both of the following inequalities.  
	
	%\begin{multicols}{2}
	\begin{enumerate} [font=\normalfont, label=(\roman*)]
		\item$ \displaystyle 
		\frac{r_n}{q_k}+\frac{10^n}{2q_k^{2}}\leq \{\pi 10^n \}\leq\frac{r_n+1}{q_k}$, \tabto{8cm} if  $ 10^n \leq q_k,$
		\item$\displaystyle
		\frac{r_n}{q_k}+\frac{1}{q_k^{\mu-1}}\leq \left \{ \pi 10^n \right \}\leq  \frac{s_n}{q_k}+\frac{c_n}{q_k^2},$ \tabto{8cm} if  $ 10^n > q_k,$
	\end{enumerate}
	%\end{multicols}
	where $\mu=\mu(\pi)\geq 2$ is the irrationality measure of $\pi$, and $0\leq c_n,r_n,s_n< q_k$. 
\end{lem}

\begin{proof} The verification is split into two cases, depending on the magnitude of the integer $10^n$.\\
	
	\textbf{Case I}. Observe that for all sufficiently large $q_k$, and any integer $10^n\leq q_k$, there is the rational approximation inequality 
	\begin{equation}\label{eq5588.110} 
		\frac{1}{2q_k^{2}}\leq \pi - \frac{p_k}{q_k}   \leq \frac{1}{q_k^{2}}\leq  \frac{1}{10^nq_k},
	\end{equation}
	for any even index $k\geq 1$, see \cite[Theorem 163]{HW1975}. This leads to the new inequality
	\begin{equation}\label{eq5588.120} 
		\frac{10^n}{2q_k^{2}}\leq \pi10^n - \frac{10^np_k}{q_k}   \leq  \frac{1}{q_k}.
	\end{equation}
	Now, replace $10^np_k=a_nq_k+r_n$ to get the equivalent expression
	\begin{equation}\label{eq5588.130} 
		\frac{10^n}{2q_k^{2}}\leq  \pi10^n - a_n-\frac{r_n}{q_k}  \leq  \frac{1}{q_k}.
	\end{equation}
	Clearly, this implies that the fraction part satisfies 
	\begin{equation}\label{eq5588.140} 
		\frac{r_n}{q_k}+\frac{10^n}{2q_k^{2}}\leq \left \{ \pi10^n \right \}\leq  \frac{r_n+1}{q_k},
	\end{equation}
	.\\
	
	\textbf{Case II}. For all sufficiently large $q_k$, and any integer $10^n> q_k$, there is the rational approximation inequality 
	\begin{equation}\label{eq5588.150} 
		\frac{1}{q_k^{\mu}} \leq \pi - \frac{p_k}{q_k}   \leq \frac{1}{q_k^{2}},
	\end{equation}
	for any even index $k\geq 1$, see \cite[Theorem 163]{HW1975}. Here, the quantity $\mu=\mu(\pi)\geq2$ is the irrationality measure of $\pi$, see \cite[p. 556]{WM2000}. This leads to the new inequality 
	\begin{equation}\label{eq5588.160} 
		\frac{1}{10^nq_k^{\mu-1}}\leq \pi - \frac{p_k}{q_k}   \leq \frac{1}{q_k^{2}}.
	\end{equation}
	Equivalently, this is 
	\begin{equation}\label{eq5588.170} 
		\frac{1}{q_k^{\mu-1}}\leq \pi10^n - \frac{10^np_k}{q_k}   \leq  \frac{10^n}{q_k^2}.
	\end{equation}
	Now, rearrange it as
	\begin{equation}\label{eq5588.175} 
		\frac{p_k10^n}{q_k}+\frac{1}{q_k^{\mu-1}}\leq \pi10^n  \leq  \frac{(p_kq_k+1)10^n}{q_k^2},
	\end{equation}
	and replace $p_k10^n=a_nq_k+r_n$ and $(p_kq_k+1)10^n=b_nq_k^2+s_nq_k+c_n$ to get the equivalent expression
	\begin{equation}\label{eq5588.180} 
		a_n+\frac{r_n}{q_k}+\frac{1}{q_k^{\mu-1}}\leq \pi10^n  \leq b_n+ \frac{s_n}{q_k}+ \frac{c_n}{q_k^2},
	\end{equation}
	where $0\leq a_n,b_n,c_n,r_n,s_n< q_k$. This implies that the fraction part satisfies 
	\begin{equation}\label{eq5588.190} 
		\frac{r_n}{q_k}+\frac{1}{q_k^{\mu-1}}\leq \left \{ \pi10^n \right \}\leq  \frac{s_n}{q_k}+\frac{c_n}{q_k^2},
	\end{equation}
	as claimed.
\end{proof}

%conditional
\begin{lem} \label{lem5588.200} Let $p_k/q_k$ be the $k$th convergent of the real number $\pi$, and let $q= q_k+o(q_k)$ be a large prime. Then, for each $n\leq q$, the fractional part $\{\pi 10^n \}$ of the real number $\pi10^n$ satisfies one or both of the following inequalities for infinitely many primes $q$.  
	%\begin{multicols}{2}
	\begin{enumerate} [font=\normalfont, label=(\roman*)]
		\item
		$ \displaystyle 
		\frac{r_n}{2q}+O\left(\frac{1}{q^{2}}\right)\leq \left \{ \pi10^n \right \}\leq  \frac{r_n+1}{q}+O\left(\frac{1}{q^{2}}\right)$, \tabto{10cm} if  $ 10^n \leq q_k,$
		\item
		$\displaystyle
		\frac{r_n}{2q}+\frac{1}{(2q)^{\mu-1}}+O\left(\frac{1}{q^{2}}\right)\leq \left \{ \pi10^n \right \}\leq  \frac{s_n}{q}+O\left(\frac{1}{q^{2}}\right)$, \tabto{10
			cm} if  $ 10^n > q_k,$
	\end{enumerate}
	%\end{multicols}
	where $\mu=\mu(\pi)\geq 2$ is the irrationality measure of $\pi$, and $0\leq c_n,r_n,s_n< q$. 
\end{lem}

\begin{proof} For a large integer $q_k\geq 2$, and a large prime $q= q_k+o(q_k)$, the inequalities
	\begin{equation}\label{eq5588.210}
		\frac{1}{2q}+O\left(\frac{1}{q^{2}}\right)\leq \frac{1}{ q_k}\leq \frac{1}{q}+O\left(\frac{1}{q^{2}}\right)
	\end{equation}
	are valid.\\
	
	\textbf{Case I}. Replacing \eqref{eq5588.210} into \eqref{eq5588.140} yields
	\begin{equation}\label{eq5588.220} 
		\frac{r_n}{2q}+O\left(\frac{1}{q^{2}}\right)\leq \left \{ \pi10^n \right \}\leq  \frac{r_n+1}{q}+O\left(\frac{1}{q^{2}}\right)
	\end{equation}
	for any even index $k\geq 1$, and all sufficiently large $q= q_k+o(q_k)$. \\
	
	\textbf{Case II}. Replacing \eqref{eq5588.210} into \eqref{eq5588.190} yields 
	\begin{equation}\label{eq5588.230} 
		\frac{r_n}{2q}+\frac{1}{(2q)^{\mu-1}}+O\left(\frac{1}{q^{2}}\right)\leq \left \{ \pi10^n \right \}\leq  \frac{s_n}{q}+O\left(\frac{1}{q^{2}}\right),
	\end{equation}
	for any even index $k\geq 1$, see \cite[Theorem 163]{HW1975}, and the irrationality measure $\mu=\mu(\pi)\geq2$ of $\pi$, 
	see \cite[p. \ 556]{WM2000} for additional information.
\end{proof}
%SSSSSSSSSSSSSSSSSSSSSSSSSSSSSSSSSSSSSSSSSSSSSSSSSSSS
%SSSSSSSSSSSSSSSSSSSSSSSSSSSSSSSSSSSSSSSSSSSSSSSSSSSS
%SSSSSSSSSSSSSSSSSSSSSSSSSSSSSSSSSSSSSSSSSSSSSSSSSSSS
%SSSSSSSSSSSSSSSSSSSSSSSSSSSSSSSSSSSSSSSSSSSSSSSSSSSS 
\section{Multiplicative Subgroups And Exponential Sums}\label{S7711}

The cardinality of a subset of integers $H\subset \mathbb{Z}$ is denoted by $\#H\geq0$. The multiplicative order of an element $r\ne0$ in a finite ring $\left(\mathbb{Z}/m\mathbb{Z}\right)^{\times}$ is defined by $\ord_m r=\min \{ n\geq 1:r^n\equiv 1\bmod m\}$, where $m\geq1$ is an integer. The multiplicative order is a divisor $N\mid \varphi(m)$ of the totient function $\varphi(m)=\prod_{p\mid m}(1-1/p)$, see \cite[Theorem 2.4]{AP1976}.

\begin{lem}\label{lem7711.060} Let $\{p_k/q_k: k\geq 1\}$ be the sequence of convergents of the real number $\pi$. Assume that $\gcd(10,q_k)=1$. Then,
	\begin{enumerate}[font=\normalfont, label=(\roman*)]
		\item $\displaystyle G=\{p_k 10^n \equiv r_n\bmod q_k: n\geq 1\}$ is a large multiplicative subgroup of the finite ring 
		$\left(\mathbb{Z}/q_k\mathbb{Z}\right)^{\times}$.
		\item $\displaystyle H=\{\left(p_kq_k+1\right)10^n \equiv s_n\bmod q_k: n \geq1\}$ is a large multiplicative subgroup of the finite ring 
		$\left(\mathbb{Z}/q_k\mathbb{Z}\right)^{\times}$.
	\end{enumerate}
\end{lem}

\begin{proof} (ii) The hypothesis $\gcd(10,q_k)=1$ implies that the integer $10$ generates a multiplicative subgroup
	\begin{equation}\label{eq7711.110}
		\mathcal{H}=\{10^0,10^1,10^2,\ldots,10^n,\ldots\}\subseteq \left(\mathbb{Z}/q_k\mathbb{Z}\right)^{\times}    
	\end{equation}
	of cardinality $N=\#\mathcal{H}$. Moreover, since $\gcd(p_kq_k+1,q_k)=1$, the map 
	\begin{equation}\label{eq7711.115}
		n\longrightarrow \left(p_kq_k+1\right)10^n \equiv s_n\bmod q_k 
	\end{equation}
	is $1-$to$-1$ on the finite ring $\left(\mathbb{Z}/q_k\mathbb{Z}\right)^{\times}$, it permutes the finite ring $\mathbb{Z}/q_k\mathbb{Z}$. Therefore, the two subsets are equal, that is, $H=\mathcal{H}$.
\end{proof}

Trivially, $N=\#\mathcal{H}>\log q_k$, but it requires considerable more works to show that 
$N> q_k^{\varepsilon}$, where $\varepsilon>0$ is a small number. The ideal case has a large prime $q=q_k\geq2$, and the integer $10$ generates the maximal multiplicative group $\mathbb{F}_q^{\times}$ of cardinality $N= \varphi(q)$.

\begin{lem}\label{lem7711.160}  Assume the Artin primitive root conjecture is valid. Let $\{p_k/q_k: 
	k\geq 1\}$ be the sequence of convergents of the real number $\pi$, and let $q\sim q_k$ be a large prime. Then,
	\begin{enumerate}[font=\normalfont, label=(\roman*)]
		\item $\displaystyle \mathcal{G}=\{p_k 10^n \equiv r_n\bmod q:n\geq0\}$ is the multiplicative group of the finite field
		$\mathbb{F}_q^{\times}$,
		\item $\displaystyle \mathcal{H}=\{\left(p_kq_k+1\right)10^n \equiv s_n\bmod q:n\geq 0\}$ is the multiplicative group of the finite field
		$\mathbb{F}_q^{\times}$,
	\end{enumerate}
	for infinitely many large primes $q= q_k+o(q_k)$ such that $\gcd(p_kq_k+1,q)=1$, as $q_k \to \infty$.
\end{lem}

\begin{proof} (ii) The conditional proof of the Artin primitive root conjecture, states that $10$ generates the multiplicative group of the finite field $\mathbb{F}_q^{\times}$ of a subset of primes $\mathcal{Q}=\{\mbox{prime } q\geq2: \ord_q 10=q-1\}$ of density $\delta(\mathcal{Q})=0.3739558\ldots$, see \cite[p.\ 220]{HC1967}. In particular, the interval $[q_k,q_k+O(q_k/\log q_k)]$ contains approximately 
	\begin{equation}\label{eq7711.380}
		\pi(q_k+O(q_k/\log q_k))-\pi(q_k)\gg q_k/(\log q_k)^2,
	\end{equation}
	large primes $q= q_k+o(q_k)$ such that $< 10>\;=\;\mathbb{F}_q^{\times}$, as $\to \infty$, confer \cite[p. \ 113]{DH1980} or similar reference. Now, proceed to use the same analysis as in the previous Lemma to complete the proof.
\end{proof}

%SSSSSSSSSSSSSSSSSSSSSSSSSSSSSSSSSSSSSSSSSSSSSSSSSSSS
%SSSSSSSSSSSSSSSSSSSSSSSSSSSSSSSSSSSSSSSSSSSSSSSSSSSS
%SSSSSSSSSSSSSSSSSSSSSSSSSSSSSSSSSSSSSSSSSSSSSSSSSSSS
%SSSSSSSSSSSSSSSSSSSSSSSSSSSSSSSSSSSSSSSSSSSSSSSSSSSS
\section{Large Multiplicative Orders Modulo $n$}\label{S0025}
The \textit{multiplicative order} of an element $u$ in a finite group $G$ of cardinality $n=\#G$ is defined by $\ord_nu=\min\{ m:u^m\equiv 1\bmod n\}$. The definition of the average multiplicative order in a fixed finite group has a very useful analytic formulation. 

\begin{dfn} \label{dfn0025.001}{\normalfont 
		Let $G$ be a cyclic group of order $n=\#G$. The average multiplicative order of the elements $u \in G$ is defined by
		\begin{equation}\label{eq0025.001}
			A(n)=\frac{1}{n}\sum_{u \in G}\ord_n u =\frac{1}{n}\sum_{d\mid  n} d \varphi(d),
		\end{equation}
		where $ \varphi(n)=n\prod_{d\mid n}(1-1/p)$.}
\end{dfn}
On average, a random element in a random finite cyclic group $G$ of cardinality $n=\#G\leq x $ has a  large order $\ord_n u \approx n$.
\begin{thm} \label{thm0025.101} The mean average order $\overline{A(n)}$ of the elements in the finite cyclic groups of cardinality $n\#G$ such that $n \leq x$ is
	\begin{equation}\label{eq0025.021}
		\overline{A(n)}=a_0x +O(\log^2 x), \nonumber
	\end{equation}
	where the constant is
	\begin{equation}\label{eq0025.023}
		a_0=\zeta(3)/2\zeta(2)=0.365381484700719249363018365653857\ldots . \nonumber
	\end{equation}
\end{thm}

\begin{proof} Taking the mean value of the average multiplicative order gives
	\begin{eqnarray}\label{eq0025.003}
		\frac{1}{x}\sum_{n\leq x}A(n)&=&\frac{1}{x}\sum_{n\leq x}\frac{1}{n}\sum_{d\mid  n} d \varphi(d) \\
		&=&\frac{1}{x}\sum_{m\leq x}\frac{1}{m}\sum_{d\leq x/m} \varphi(d) \nonumber.
	\end{eqnarray}
	The condition $d \mid n$ was used to cancel the $d$ term in  the inner sum. Substituting the asymptotic average order of the totient function $\varphi(n)$, \cite[Theorem 3.7]{AP1976}, and similar references,  leads to 
	\begin{eqnarray}\label{eq0025.005}
		\frac{1}{x}\sum_{n\leq x}A(n)&=&\frac{1}{x}\sum_{m\leq x}\frac{1}{m}\left (\frac{1}{2 \zeta(2)}\left ( \frac{x}{m}\right )^2+O\left (\frac{x}{m}\log (x/m)\right ) \right ) \\
		&=&\frac{x}{2 \zeta(2)}\sum_{m\leq x}\frac{1}{m^3}+O\left ((\log x)\sum_{m\leq x} \frac{1}{m}\right ) 
		\nonumber\\
		&=&\frac{\zeta(3)}{2 \zeta(2)}x+O\left (\log^2 x \right ) \nonumber.
	\end{eqnarray}
	This completes the proof.
\end{proof}
The double averaging accounts for the small error term. Moreover, since the mean average order $\overline{A(n)}$ is almost the same magnitude as the largest groups $\#G=n \approx x$, this result shows that the generators of the cyclic group, elements of maximal multiplicative orders, contribute the sheer bulk of the of the mean average order. A slightly more difficult proof appears in \cite[Theorem 3.1]{VS2004}.\\

Almost every element in a random finite cyclic group $G$ of cardinality $n=\#G\leq x $ has a  large order bounded below by $\ord_n u \gg n/\log n$.

\begin{thm} \label{thm0025.201} {\normalfont (\cite[Theorem 6]{KP2013})} Assume GRH. Let $ x$ be a large number, and let $G$ be finite cyclic group of cardinality $n=\#G \leq x$. Then, almost every element $u\in G$ has large multiplicative order $\ord_n u\gg n/\log x$. 
\end{thm}
%SSSSSSSSSSSSSSSSSSSSSSSSSSSSSSSSSSSSSSSSSSSSSSSSSSSS
%SSSSSSSSSSSSSSSSSSSSSSSSSSSSSSSSSSSSSSSSSSSSSSSSSSSS
%SSSSSSSSSSSSSSSSSSSSSSSSSSSSSSSSSSSSSSSSSSSSSSSSSSSS
%SSSSSSSSSSSSSSSSSSSSSSSSSSSSSSSSSSSSSSSSSSSSSSSSSSSS
%\newpage
\section{Estimates For Exponential Sums} \label{S7700}
\begin{thm}\label{thm7700.460} {\normalfont (\cite[Theorem 1]{BG2011})} Let $\mathcal{H}\subseteq \mathbb{F}_p$ be a multiplicative subgroup of order $\# \mathcal{H} > p^{c_0/\log \log p}$ for some sufficiently large constant $c_0 > 1$. Then
	
\begin{equation}\label{eq7700.400}
\max_{\gcd(a,p) =1} \sum_{x \in \mathcal{H}} e^{i2\pi a x/p} < e^{-(\log p)^c} \# \mathcal{H},
\end{equation}
where $c> 0$ is an absolute constant.
\end{thm}

%SSSSSSSSSSSSSSSSSSSSSSSSSSSSSSSSSSSSSSSSSSSSSSSSSSSS
%SSSSSSSSSSSSSSSSSSSSSSSSSSSSSSSSSSSSSSSSSSSSSSSSSSSS
%SSSSSSSSSSSSSSSSSSSSSSSSSSSSSSSSSSSSSSSSSSSSSSSSSSSS
%SSSSSSSSSSSSSSSSSSSSSSSSSSSSSSSSSSSSSSSSSSSSSSSSSSSS
%\newpage
\section{Properties of the Exponential Function} \label{S7770}

\begin{lem}\label{lem7770.800} Let $\{p_k/q_k: k\geq 1\}$ be the sequence of convergents of the irrational real number $\alpha\ne0$, and let $q= q_k+o(q_k)$ be a large integer. If the fractional part has an effective rational approximation
	
	\begin{equation}\label{eq7770.800}
		\left |\{\alpha10^n\}-\frac{s_n}{q}\right|  \ll   \frac{1}{q^2},
	\end{equation}
	where $0\leq s_n<q_k$, then
	\begin{equation}\label{eq7770.810}
		\left |e^{i2\pi \{\alpha10^n\}}-e^{i2\pi\frac{s_n}{q}}\right|  \ll   \frac{1}{q^2}.
	\end{equation}
\end{lem}

\begin{proof}[\textbf{Proof}] Basically, this follows from the Lipschitz property
	\begin{equation}\label{eq7770.820}
		\left |f(x)-f(y)\right|  \ll   \left |x-y\right| 
	\end{equation}
	of the continuous function $f(x)=e^{ix}$ of the real variable $0\leq |x|<1$. Specifically,
	
	\begin{eqnarray}\label{eq7770.830}
		D&=& \left |e^{i2\pi\{\alpha10^n\}}-
		e^{i2\pi\frac{s_n}{q}}\right|  \\
		&=&\left | e^{i2\pi\frac{s_n}{q}}\left(e^{i2\pi  \left(\{\alpha10^n\}-\frac{s_n}{q}\right)}-1\right)\right| \nonumber\\
		&=&
		\left | 
		\cos 2\pi\left(\{\alpha10^n\}-\frac{s_n}{q}\right)-1
		+i\sin 2\pi\left( \{\alpha10^n\}-\frac{s_n}{q}\right)\right|  \nonumber\\
		& \leq& \left | \cos 2\pi\left( \{\alpha10^n\}-\frac{s_n}{q}\right)-1\right|+
		\left |\sin 2\pi\left( \{\alpha10^n\}-\frac{s_n}{q}\right)\right| \nonumber\\
		& \leq& 
		\left |\sin 2\pi\left( \{\alpha10^n\}-\frac{s_n}{q}\right)\right | \nonumber\\
		& \leq& 
		\left | \{\alpha10^n\}-\frac{s_n}{q}\right | \nonumber\\
		&\ll&\frac{1}{q^2} ,\nonumber
	\end{eqnarray}
	since $\cos z=1+O(z^2)$ and $\sin z=z+O(z^3)$ for $0\leq |z|<1$.
\end{proof}

%SSSSSSSSSSSSSSSSSSSSSSSSSSSSSSSSSSSSSSSSSSSSSSSSSS
%SSSSSSSSSSSSSSSSSSSSSSSSSSSSSSSSSSSSSSSSSSSSSSSSSS
%SSSSSSSSSSSSSSSSSSSSSSSSSSSSSSSSSSSSSSSSSSSSSSSSSS
%SSSSSSSSSSSSSSSSSSSSSSSSSSSSSSSSSSSSSSSSSSSSSSSSSS
\section{Exponential Sums and Liouville Numbers}\label{S3007}
The next lemma demonstrates how the properties of irrational numbers can change the estimates of exponential sums. 

\begin{thm}\label{thm3007.400}{\normalfont (Weyl)}  Let $f(t)=a_dt^d+a_{d-1}t^{d-1}+\cdots+a_1t+a_0$ be a polynomial of  degree $\deg f=d$, and let $H$, $a$, $q\geq1$ be integers with $\gcd(a, q) = 1$. If $\alpha$ is an irrational number such that 
	\begin{equation}\label{eq3007.400}
		\left|\alpha-\frac{a}{q} \right|<\frac{c}{q^\mu}
	\end{equation} 
	for some constants $c\geq1$ and $\mu\geq2$, then for any small number $\varepsilon>0$,
	\begin{equation}\label{eq3007.405}
		\sum_{H\leq n\leq H+x}e^{i2\pi k f(n)\alpha}\ll x^{1+\varepsilon}\left( \frac{c}{q}+\frac{1}{x}+\frac{q}{x^{d}}\right)^{2^{1-d}} 
	\end{equation}
\end{thm}

The combination of the arbitrary large blocks of consecutive zeros in the $b$-adic expansion of a Liouville number, and Weyl inequality proves that certain exponential sums have trivial upper bounds. 
\begin{lem}\label{lem3007.300}Let $f(t)$ be a polynomial of degree $\deg f=d\geq1$, and let $\alpha$ be a Liouville number. Then, 
	$$	 \sum_{ n \leq x}  e^{i2\pi k f(n)\alpha}=\Omega_{\pm}(x) ,$$	
	where the implied constant depends on the irrational number $\alpha$, and the parameter $k\ne0$ as $x\to \infty$.
\end{lem}

\begin{proof}[\textbf{Proof}] Let $\{p_v/q_v:v\geq1\}$ be the sequence of convergents of the Liouville $\alpha$, and consider
	\begin{equation}\label{eq3007.310}
		\left | \alpha -\frac{p_v}{q_v}\right |> \frac{1}{q_v^{\mu+\varepsilon}}.
	\end{equation}
	Next, let $x^{d-1}\in [q_v,q_v^{\mu}]$. Replacing these in the Weyl inequality, Theorem \ref{thm3007.400}, yields
	\begin{eqnarray}\label{eq3007.330}
		\sum_{H\leq n\leq H+x}e^{i2\pi k f(n)\alpha}&\ll& x^{1+\varepsilon}\left( \frac{c}{q}+\frac{1}{x}+\frac{q}{x^{d}}\right)^{2^{1-d}} \\
		&\ll& x^{1+\varepsilon}\left( \frac{c}{x^{\frac{d}{\mu+\varepsilon}}}+\frac{1}{x}+\frac{x^{d-1}}{x^{d}}\right)^{2^{1-d}}\nonumber,	
	\end{eqnarray}
	$c\geq1$ and $k\ne0$ are constants.  
	Therefore, the hypothesis $\mu \to\infty$ implies the trivial inequality
	\begin{equation}\label{eq30073240}
		\sum_{H\leq n\leq H+x}e^{i2\pi k f(n)\alpha}\ll x^{1+\varepsilon}.
	\end{equation}
	This proves the claim.
\end{proof}

%SSSSSSSSSSSSSSSSSSSSSSSSSSSSSSSSSSSSSSSSSSSSSSSSSSSS
%SSSSSSSSSSSSSSSSSSSSSSSSSSSSSSSSSSSSSSSSSSSSSSSSSSSS
%SSSSSSSSSSSSSSSSSSSSSSSSSSSSSSSSSSSSSSSSSSSSSSSSSSSS
%SSSSSSSSSSSSSSSSSSSSSSSSSSSSSSSSSSSSSSSSSSSSSSSSSSSS
%\newpage
\section{Conditional Proof For The Normality Of $\sqrt{2}$} \label{S8877}

The proof of the normality of $\sqrt{2}$ in base $10$ is based on the rational prime approximations in stated in Section \ref{S4466}, the average large subgroups modulo $p$ generated by the base $10$, see Theorem \ref{thm0025.101}, (which is similar to the conditional result for rational primes approximations in Section \ref{S4466}), the Weyl criterion in Section \ref{S2266}, and the Wall criterion stated in Theorem \ref{thm2270.150}. 

\begin{proof}[\textbf{Proof}] {\bfseries (Theorem \ref{thm2200.300})} Let $x$ be a large number. Now, consider the sequence of real numbers 
	\begin{equation}\label{eq2266.200}
		\mathscr{B}(\sqrt{2})=\{ \sqrt{2}\cdot  10^n: n\geq 1\},
	\end{equation}
	and the corresponding exponential sum
	\begin{equation}\label{eq2266.205}
		\sum_{n\leq x}e^{i2 \pi r \sqrt{2}\cdot  10^n },
	\end{equation}
	where $r\ne0$ is a rational parameter. By Theorem \ref{thm2266.600}, there exists an infinite sequence of primes $\mathcal{P}= \{p\geq x:p \text{ prime}\}$ such that 
	\begin{equation}\label{eq2266.210}
		\left | \sqrt{2}-\frac{a}{p}\right |< \frac{1}{p^{4/3-\varepsilon} }
	\end{equation}
	for each $p\in\mathcal{P} $, where $1\leq a <p$. Moreover, on average, the subset of integers
	\begin{equation}\label{eq2266.220}
		\mathcal{A}=\{ 10^n\equiv m \bmod p: n\leq x\leq p\}
	\end{equation}
	is sufficiently large for each $p\in\mathcal{P} $, and has cardinality $\#\mathcal{A}\gg p^{\delta}$, where $\delta>0$ is a small number,  see Theorem \ref{thm0025.101} and Conjecture \ref{conj2266.700}. Therefore, 
	\begin{eqnarray}\label{eq2266.230}
		\sum_{n\leq x}e^{i2 \pi r \sqrt{2}\cdot  10^n }&=&
		\sum_{m\leq x}e^{i2 \pi r \left ( \frac{a}{p}+O\left (\frac{1}{p^{4/3-\varepsilon}}\right )\right ) \cdot  10^n }\\
		&=&\sum_{m\in\mathcal{A} }e^{i2 \pi r \left ( \frac{am}{p}+O\left (\frac{m}{p^{4/3-\varepsilon}}\right )\right )  }\nonumber.
	\end{eqnarray}
	By Theorem \ref{thm7700.460}, the exponential sums have nontrivial upper bounds
	\begin{equation}\label{eq7799.950}
		\sum_{n\leq x}e^{i2 \pi r \sqrt{2}\cdot  10^n }=\sum_{m\in\mathcal{A} }e^{i2 \pi r \left ( \frac{am}{p}+O\left (\frac{m}{p^{4/3-\varepsilon}}\right )\right )  }\ll p^{1-\varepsilon}.
	\end{equation}
	By the Weyl criterion, see Theorem \ref{thm2266.500}, it follows that the sequence of real numbers \eqref{eq2266.200} is uniform distributed modulo $1$. Lastly, by Theorem  \ref{thm2270.150}, it follows that the real number $\sqrt{2}$ is normal in base $b=10$.
\end{proof}

%SSSSSSSSSSSSSSSSSSSSSSSSSSSSSSSSSSSSSSSSSSSSSSSSSSSS
%SSSSSSSSSSSSSSSSSSSSSSSSSSSSSSSSSSSSSSSSSSSSSSSSSSSS
%SSSSSSSSSSSSSSSSSSSSSSSSSSSSSSSSSSSSSSSSSSSSSSSSSSSS
%SSSSSSSSSSSSSSSSSSSSSSSSSSSSSSSSSSSSSSSSSSSSSSSSSSSS
%\newpage
\section{Conditional Proof For The Normality Of Pi}\label{S7799}

The proof is based on the foundational results in Section \ref{S5588}, Section \ref{S7711}, the Weyl criterion in Section \ref{S2266}, and the Wall criterion stated in Theorem \ref{thm2270.150}.

\begin{proof}[\textbf{Proof}] {\bfseries (Theorem \ref{thm2200.400})} By Lemma \ref{lem5588.200}, almost every fractional part has the random rational approximation 
	\begin{equation}\label{eq7799.900}
		\frac{r_n}{2q}+\frac{1}{(2q)^{\mu-1}}+O\left(\frac{1}{q^{2}}\right)\leq \left \{ \pi10^n \right \}\leq  \frac{s_n}{q}+O\left(\frac{1}{q^{2}}\right),
	\end{equation}
	where $0\leq c_n,r_n,s_n\leq q_k$, and $\mu\geq2$. There are at most $O(\log q)$ exceptions, see Lemma \ref{lem5588.100}. The nonsymmetric inequalities \eqref{eq7799.900} are rewritten as
	
	\begin{equation}\label{eq7799.902}
		\left |\left \{ \pi10^n \right \} -\frac{r_n}{2q}-\frac{1}{(2q)^{\mu-1}}\right | \ll \frac{1}{q^{2}},
	\end{equation}
	and
	
	\begin{equation}\label{eq7799.904}
		\left |\left \{ \pi10^n \right \} -\frac{s_n}{q}\right | \ll\frac{1}{q^{2}},
	\end{equation}
	
	Therefore, by Lemma \ref{lem7770.800}, the corresponding exponentials pairing
	\begin{equation}\label{eq7799.910}
		e^{i2 \pi m\left ( \frac{r_n}{2q}+\frac{1}{(2q)^{\mu-1}}\right )}\asymp e^{i2 \pi m\left \{ \pi10^n \right \}},
	\end{equation}
	and
	\begin{equation}\label{eq7799.912}
		e^{i2 \pi m\left \{ \pi10^n \right \}}\asymp  e^{i2 \pi m \frac{s_n}{q}},
	\end{equation}
	where $m\ne0$, are proportionals. Similarly, the corresponding exponential sums
	\begin{equation}\label{eq7799.920}
		\sum_{n\leq q}e^{i2 \pi m\left ( \frac{r_n}{2q}+\frac{1}{(2q)^{\mu-1}}\right )}\asymp \sum_{n\leq q}e^{i2 \pi m\left \{ \pi10^n \right \}},
	\end{equation}
	and
	\begin{equation}\label{eq7799.922}
		\sum_{n\leq q}e^{i2 \pi m\left \{ \pi10^n \right \}}\asymp  \sum_{n\leq q}e^{i2 \pi m \frac{s_n}{q}},
	\end{equation}
	where $q= q_k+o(q_k)$ is a large prime such that $\gcd(p_kq_k+1,q)=1$, are proportionals. \\
	
	By Lemma \ref{lem7711.160} the subsets 
	\begin{equation}\label{eq7799.930}
		G=\{p_k10^n\equiv r_n \bmod q\} \quad \text{ and } \quad H=\{(p_kq_k+1)10^n\equiv s_n \bmod q\}
	\end{equation}
	are sufficiently large multiplicative subgroups of the finite field 
	$\mathbb{F}_q^{\times}$. In particular, conditional on the Artin primitive root conjecture, the cardinalities are $\# G = q-1\gg q^{\varepsilon}$, and $\# H = q-1\gg q^{\varepsilon}$, where 
	$\varepsilon>0$.\\
	
	By Theorem \ref{thm7700.460}, the exponential sums have nontrivial upper bounds
	\begin{equation}\label{eq7799.940}
		\sum_{n\leq q}e^{i2 \pi m\left ( \frac{r_n}{2q}+\frac{1}{(2q)^{\mu-1}}\right )}\asymp\sum_{n\leq q}e^{i2 \pi m\left \{ \pi10^n \right \}}\ll q^{1-\varepsilon},
	\end{equation}
	and
	\begin{equation}\label{eq7799.942}
		\sum_{n\leq q}e^{i2 \pi m\left \{ \pi10^n \right \}}\asymp  \sum_{n\leq q}e^{i2 \pi m\frac{s_n}{q}}\ll q^{1-\varepsilon}.
	\end{equation}
	By the Weyl criterion, see \cite[Theorem 2.1]{KN1974}, any of the expressions \eqref{eq7799.940} or \eqref{eq7799.942} is sufficient to prove the uniform distribution of the sequence $\{\pi 10^n:n \geq 1\}$. 
\end{proof}

Establishing the main result as an unconditional does not seem to be difficult, because proving the existence of infinitely many large multiplicative subgroups 
\begin{equation}\label{eq7799.950}
	\mathcal{H}=\{10^0,10^1,10^2,\ldots,10^n,\ldots\}\subseteq \left(\mathbb{Z}/q_k\mathbb{Z}\right)^{\times}  
\end{equation}

of cardinalities $\#\mathcal{H} \gg q^{\varepsilon}$ as $q_k \to \infty$, is not a difficult task, see Section \ref{S0025}. The techniques employed appear to be extendable to other irrational numbers $\alpha$ of finite irrationality measure $\mu(\alpha)\geq2$ and 
base $b=10$. The generalization to other bases $b\ne10$ seems to require significant additional works.

%SSSSSSSSSSSSSSSSSSSSSSSSSSSSSSSSSSSSSSSSSSSSSSSSSSSS
%SSSSSSSSSSSSSSSSSSSSSSSSSSSSSSSSSSSSSSSSSSSSSSSSSSSS
%SSSSSSSSSSSSSSSSSSSSSSSSSSSSSSSSSSSSSSSSSSSSSSSSSSSS
%SSSSSSSSSSSSSSSSSSSSSSSSSSSSSSSSSSSSSSSSSSSSSSSSSSSS
%\newpage
\section{Results for Multiplicative Nonlacunary Sequences}\label{S9900}
The proof provided here is based on the results for \textit{"multiplicative orbits"} of nonlacunary sequences \eqref{eq9900.903}. This proof seems to be depend on the arithmetic properties of the irrational numbers. Accordingly, this analysis seems to be well suited for structured irrational numbers such as algebraic irrational numbers, and nonLiouville numbers. This restriction circumvents the exceptional cases such as the Liouville type numbers 
\begin{equation}\label{eq9900.900}
	\alpha=	\sum_{ n\geq 1}\frac{1}{b^{ f(n)}},
\end{equation}  
where $b>1$ is a base, and $f(n)$ is a rapidly increasing function. In \cite[p.\ 49]{FH1967} there is a short discussion on an exceptional case, and confer \cite{MG2001} for some other exceptional cases. Other closely related problems are studied in \cite{AT2004}, \cite{DR2008}, et alii.

\begin{thm} \label{thm9900.900} Let $p\geq2$ and $q\geq2$ be multiplicative independent integers, and let $\alpha $ be a nonLiouville number. Suppose that the sequence of real numbers
	\begin{equation}\label{eq9900.903}
		\{p^nq^m\alpha:m,n\geq1\}.
	\end{equation}
	is dense in $\R/\Z$, then, the irrational number $\alpha $ is a normal number in both base $p$ and base $q$. 
\end{thm}

\begin{proof}[\textbf{Proof}] (i) Let $k=q^m$ be an integer parameter, and consider the sequence $\beta_n=p^n\alpha$, with $m,n\geq1$. \\
	
	It is sufficient to consider nonLiouville irrational number numbers $\alpha $. Otherwise, for Liouville numbers $\alpha $, equation \eqref{eq9900.902} is false since the $p$-adic expansion of the number $\alpha$ can have blocks of consecutive zeros of arbitrary lengths infinitely often as $x\to \infty$. So Liouville numbers are not normal in any base $p$, see Remark \ref{rmk9900.900}.\\
	
	Suppose that the nonLiouville irrational number $\alpha$ is not a normal number in base $p$. Equivalently, the Wall criterion
	\begin{equation}\label{eq9900.902}
		\frac{1}{x}\;	\sum_{ n\leq x}e^{i2 \pi k\beta_{n}}=\frac{1}{x}\;	\sum_{ n\leq x}e^{i2 \pi \left(p^nq^m\alpha\right )}=o(1) 
	\end{equation}
	any parameter $k=q^m>0$, as $x\to\infty$, is false, see Theorem \ref{thm2270.150}. \\
	
This implies that  
\begin{equation}\label{eq9900.910}
		\frac{1}{x}\;	\sum_{ n\leq x}e^{i2 \pi k\beta_{n}}=\frac{1}{x}\;	\sum_{ n\leq x}e^{i2 \pi \left(p^nq^m\alpha\right )}=\Omega_{\pm}(1)  
	\end{equation}
	as $x\to\infty$, see the definition of the symbol $o(1)$ in \eqref{eq2266N.100} and its negation in \eqref{eq2266N.115}.\\
	
	 Moreover, summing over the parameter $k=q^m\leq x$ yields
	\begin{eqnarray}\label{eq9900.920}
		\frac{1}{x}\;\sum_{ m\leq x} \frac{1}{x}\;	\sum_{ n\leq x}e^{i2 \pi k\beta_{n}}&=&\frac{1}{x}\;\sum_{ m\leq x}	\frac{1}{x}\;\sum_{ n\leq x}e^{i2 \pi \left(p^nq^m\alpha\right )}\\
		&=&\frac{1}{x}\;\sum_{ m\leq x}	\Omega_{\pm}(1) \nonumber\\
		&=&	\Omega_{\pm}(1) \nonumber
	\end{eqnarray}
	as $x\to\infty$. But, this contradicts the hypothesis that the sequence of real numbers
	\begin{equation}\label{eq9900.930}
		\{p^nq^m\alpha:m,n\geq1\}.
	\end{equation}
	is dense in $\R/\Z$. Specifically, by Lemma \ref{lem3377.750}, there is a rearrangement of this sequence 
	\begin{equation}\label{eq9900.935}
		\{\beta_{m,n}=p^{\sigma(n)}q^{\sigma(m)}\alpha:m,n\geq1\}
	\end{equation}	
	such that
	\begin{eqnarray}\label{eq9900.940}
		\frac{1}{x^2}\; \sum_{ m\leq x,} 	\sum_{ n\leq x}e^{i2 \pi h\beta_{m,n}}&=&\frac{1}{x^2}\;\sum_{ m\leq x,}	\sum_{ n\leq x}e^{i2 \pi h \left(p^{\sigma(n)}q^{\sigma(m)}\alpha\right )}\\&=&o(1)\nonumber,
	\end{eqnarray}
where $h\ne0$, as $x\to\infty$. Therefore, the hypothesis \eqref{eq9900.910} is false. It implies that the irrational number $\alpha$ is a normal number in base $p\geq2$.\\
	
	(ii) Since the sequence \eqref{eq9900.930} is symmetric in the bases $p$ and $q$, reversing the labels in (i) yields the same result for base $q\geq2$. 
\end{proof}

\begin{rmk}\label{rmk9900.900}{\normalfont 
		Note that for Liouville type numbers $\alpha$, the $p$-adic expansions have arbitrary large number of zeros, so the relation
		\begin{equation}\label{eq9900.950}
			\frac{1}{x}\;	\sum_{ n\leq x}e^{i2 \pi k\beta_{n}}=\frac{1}{x}\;	\sum_{ n\leq x}e^{i2 \pi \left(p^nq^m\alpha\right )}=\Omega_{\pm}(1) 
		\end{equation}
		any parameter $k=q^m>0$, as $x\to\infty$, is actually true, see Lemma \ref{lem3007.300}. Thus, Theorem \ref{thm9900.900} is restricted to a collection of nonLiouville numbers $\alpha$. For example, quadratic irrationals cannot have arbitrary large number of zeros in the decimal expansions since the irrationality measures of these irrational numbers is 2.	
	}
\end{rmk}
\begin{rmk}{\normalfont
		The dense sets $\{\theta^n:n\geq 1\}$ generated by Pisot and Salem numbers $\theta> 1$ have a single generators. These dense sets, which are ''exponential orbits'', have different properties than the dense sets generated by two or more generators, as in \eqref{eq9900.930}. These dense sets, which are "multiplicative orbits", and other closely related problems are studied in \cite{AT2004}, \cite{DR2008}, et alii. 
	}
\end{rmk}

%SSSSSSSSSSSSSSSSSSSSSSSSSSSSSSSSSSSSSSSSSSSSSSSSSSSS
%SSSSSSSSSSSSSSSSSSSSSSSSSSSSSSSSSSSSSSSSSSSSSSSSSSSS
%SSSSSSSSSSSSSSSSSSSSSSSSSSSSSSSSSSSSSSSSSSSSSSSSSSSS
%SSSSSSSSSSSSSSSSSSSSSSSSSSSSSSSSSSSSSSSSSSSSSSSSSSSS
%\newpage
\section{Unconditional Proof For The Normality Of $\sqrt{2}$}\label{S9905}
Theorem \ref{thm2200.300} is a corollary of the previous theorem. The details are shown below, and the statement is repeated here for convenience.\\

\textbf{Theorem 1.1} \textit{The irrational number $\sqrt{2} \in \R$ is a normal number in base $p\geq2$. In particular, $\sqrt{2}$ is simply normal number in base $10$. Hence, the decimal expansion $$\sqrt{2} = 1.4142135623730950488016887242096980785696718753769480731766797379\ldots$$ 
	contains infinitely many digit $0$, infinitely many digit $1$, infinitely many digit $2$, et cetera.}

\begin{proof}[\textbf{Proof}] Fix the pair of multiplicative independent integers $p=10$, $q=3$, and let $\alpha=\sqrt{2}$. By the Furstenberg criterion, see Theorem \ref{thm3377.800}, the set of real numbers
	\begin{equation}\label{eq9900.960}
		\{3^m10^n\alpha:m,n\geq1\}.
	\end{equation}
	is dense in $\R/\Z$. Therefore, by Theorem \ref{thm9900.900}, the irrational number $\sqrt{2}$ is normal in base $p=10$.
\end{proof}

%PPPPPPPPPPPPPPPPPPPPPPPPPPPPPPPPPPPPPPPPPPPPPPPPPPPPPP
%PPPPPPPPPPPPPPPPPPPPPPPPPPPPPPPPPPPPPPPPPPPPPPPPPPPPPP
%PPPPPPPPPPPPPPPPPPPPPPPPPPPPPPPPPPPPPPPPPPPPPPPPPPPPPP
%PPPPPPPPPPPPPPPPPPPPPPPPPPPPPPPPPPPPPPPPPPPPPPPPPPPPPP
\newpage
\section{Problems}\label{exe3505}
%TTTTTTTTTTTTTTTTTTTTTTTTTTTTTTTTTTTTTTTTTT
\subsection{Lacunary and Nonlacunary Sequences}
\begin{exe}\label{exe3505.200} {\normalfont
Show that the sequence $\{u_n=2^{n}+3^{n}:n\geq 1\}$	is lacunary.
}
\end{exe}

\begin{exe}\label{exe3505.205} {\normalfont
		Show that the sequence $\{u_{m,n}=2^{m}3^{n}:m,n\geq1 \}$ is nonlacunary.
	}
\end{exe}

%TTTTTTTTTTTTTTTTTTTTTTTTTTTTTTTTTTTTTTTTTT
\subsection{Powers of Polynomials Roots}

\begin{exe}\label{exe3505.405} {\normalfont
Which of these polynomials $x^3-x+1$, $x^3-x-2$, and $x^3+x+1$ has root which is a Pisot number?
	}
\end{exe}

\begin{exe}\label{exe3505.415} {\normalfont Let $\theta>1,\theta_1=\theta^{-1},\theta_2,\ldots,\theta_9$ be the roots of the Lehmer polynomial $x^{10}+x^9-x^7-x^6-x^5-x^4-x^3+x+1$. Show the the sequence of real numbers
$$x_n\equiv\theta^n+\theta_1^n+r_1^n\cos \omega_1n+\cdots+r_4^n\cos \omega_4n \mod 1,$$ where $\theta_i+\theta_{i+1}=r_1\cos \omega_1$,  $|\theta_i|=r_i\leq1$ and $\omega_i \in (-\pi,\pi)$ for $i\in \{1,2,3,4\}$, is dense in the unit interval $(0,1)$.
	}
\end{exe}

%BBBBBBBBBBBBBBBBBBBBBBBBBBBBBBBBBBBBBBBBBBBBBBBBBBBBBB
%BBBBBBBBBBBBBBBBBBBBBBBBBBBBBBBBBBBBBBBBBBBBBBBBBBBBBB
%BBBBBBBBBBBBBBBBBBBBBBBBBBBBBBBBBBBBBBBBBBBBBBBBBBBBBB
%\newpage
%\section{References}

\vskip .21 in 
\currfilename.\\

\end{document}